\documentclass[preprint,12pt,numbers]{elsarticle}

\usepackage{lmodern}
\usepackage{amsmath,amssymb,amsthm,mathtools}
\usepackage{enumitem}
\usepackage{url}
\usepackage[colorlinks=true,linkcolor=blue,citecolor=blue,urlcolor=blue]{hyperref}

\journal{Discrete Mathematics}

\numberwithin{equation}{section}

\newtheorem{theorem}{Theorem}[section]
\newtheorem{corollary}[theorem]{Corollary}
\newtheorem{proposition}[theorem]{Proposition}

\theoremstyle{definition}

\theoremstyle{remark}
\newtheorem{remark}[theorem]{Remark}
\newtheorem{example}[theorem]{Example}

\newcommand{\Fq}{\mathbb F_q}
\newcommand{\R}{\mathbb R}
\newcommand{\Prob}{\mathbb P}
\newcommand{\E}{\mathbb E}
\newcommand{\Ind}{\mathcal I}

\newcommand{\Aut}{\operatorname{Aut}}
\newcommand{\rank}{\operatorname{rank}}
\newcommand{\PG}{\operatorname{PG}}

\newcommand{\myqed}{\hfill $\Diamond$}

\begin{document}

\begin{frontmatter}

\title{Maximum Probability of Independence in\\ Transitive Matroids}

\author[inst1]{Mladen Kova\v{c}evi\'c\corref{cor1}}
\ead{kmladen@uns.ac.rs}
\cortext[cor1]{Corresponding author.}

\affiliation[inst1]{organization={Faculty of Technical Sciences, University of Novi Sad},
country={Serbia}}

\begin{abstract}
Let $M$ be a matroid on a finite ground set $E$, and suppose that the automorphism group of $M$ acts transitively on $E$.
We show the following: if $X_1,\ldots,X_K$ are sampled independently from a distribution $p$ on $E$, then the probability that the samples are distinct and that $\{X_1,\ldots,X_K\}$ is an independent set in $M$ is quasi-concave in $p$ and maximized when $p$ is uniform.
For $2\le K\le\rank(M)$, we characterize all maximizers by uniform total mass on the parallel classes and, in particular, show that the maximizer is unique and has global quadratic stability when $M$ is simple.
As a corollary, for a random $K\times N$ matrix over a finite field with i.i.d.\ rows, uniform sampling from the nonzero vectors maximizes the probability of full row rank.
Finally, we show that uniform sampling maximizes $\Prob\big(\rank_M(X_1,\ldots,X_L)\ge r\big)$, and hence also $\E\rank_M(X_1,\ldots,X_L)$, for every $L, r$.
\end{abstract}


\end{frontmatter}

\section{Introduction}\label{sec:introduction}

Let $M$ be a finite matroid on a ground set $E$, with independent sets $\Ind(M)$~\cite{Oxley2011}. Let $X_1,\ldots,X_K\in E$ be independent samples from a probability distribution
\begin{equation}\label{eq:probability-simplex}
  p=(p_e)_{e\in E}\in \Delta(E)
  \coloneq\left\{x\in \R_{\ge 0}^{E} \colon \sum_{e\in E}x_e=1\right\}.
\end{equation}
We are interested in the probability
\begin{equation}
\label{eq:independence-probability-intro}
  F_{M,K}(p)
  \coloneq
  \mathbb P_p\bigl(|\{X_1,\ldots,X_K\}|=K
  \text{ and } \{X_1,\ldots,X_K\}\in\mathcal I(M)\bigr).
\end{equation}
Note that the event in \eqref{eq:independence-probability-intro} implies that the samples are distinct, so one can write
\begin{equation}\label{eq:independence-polynomial-relation-intro}
  F_{M,K}(p)
  =K!\sum_{\substack{S\in \Ind(M)\\ |S|=K}}\prod_{e\in S}p_e.
\end{equation}

We shall prove that the uniform distribution maximizes $F_{M,K}$ when $M$ is transitive (Theorem~\ref{thm:transitive-matroid}). In other words, among all i.i.d.\ sampling distributions on the ground set of a
transitive matroid, the uniform distribution maximizes the probability of
obtaining an independent \(K\)-set. We also study the corresponding uniqueness and stability questions.
The proof uses the complete log-concavity of matroid basis-generating
polynomials due to Anari, Oveis Gharan, and Vinzant~\cite{AnariOveisGharanVinzant2021}. The same conclusion also follows from the Lorentzian-polynomial framework of Br\"and\'en and Huh~\cite{BrandenHuh2020}, in which matroid
basis-generating polynomials form a central class of examples.
For $2\le K\le\rank(M)$, the strict log-concavity theorem of Murai, Nagaoka, and Yazawa~\cite{MuraiNagaokaYazawa2021} also gives a complete description of the maximizers: the total mass on each parallel class must be the same. In particular, every simple transitive matroid has a unique maximizer and global quadratic stability (Theorem~\ref{thm:parallel-class-maximizers}).

Our motivating example, and a special case that we analyze separately, is the projective geometry matroid $\PG(N-1,q)$. If rows of a random matrix over a finite field $\Fq$ are sampled from $\Fq^N\setminus\{0\}$, scalar multiples are indistinguishable for rank purposes. Thus the natural state space is the set of projective points
\begin{equation}\label{eq:projective-space-size}
  E=\PG(N-1,q),
  \qquad
  |E|=\frac{q^N-1}{q-1}.
\end{equation}
The event that $K$ sampled projective points are independent is precisely the event that the corresponding $K\times N$ matrix has full row rank.

Rank distributions of random matrices over finite fields have a long history, including classical and asymptotic results for uniformly distributed matrices and for sparse or biased models; see, for example, \cite{Cooper2000Distribution,Cooper2000Rank,Fulman2002,BlomerKarpWelzl1997,SalmondGrantGrivellChan2014}. Delić and Ivetić \cite{DelicIvetic2025} studied the full-rank probability for arbitrary row distributions over finite fields and proved local optimality of the uniform distribution on $\mathbb F_q^N\setminus\{0\}$ for $1\le K\le N$, as well as global optimality for $K=2$. Corollary~\ref{cor:finite-field-full-rank} below establishes global optimality for all $1\le K\le N$, resolving the corresponding question in \cite{DelicIvetic2025}. Remarks~\ref{rem:nonzero-vector-distributions} and~\ref{rem:vector-level-equality} give the vector-level formulation, allowing mass at zero, and characterize all optimal distributions for $2\le K\le N$.

In Section~\ref{sec:further-questions}, we extend the same extremality principle beyond independence: for every sample size $L$ and every rank threshold $r$, uniform sampling maximizes the probability that the sampled rank is at least $r$. The proof uses the product and support properties of Lorentzian polynomials. Thus uniform sampling maximizes the entire sampled-rank distribution in the sense of first-order stochastic dominance, and in particular maximizes expected rank.

\section{Matroid polynomials and concavity}\label{sec:matroid-polynomials}

For $K\ge 0$, define the degree-$K$ independent-set generating polynomial
\begin{equation}\label{eq:fk-definition}
  f_{M,K}(x)
  \coloneq\sum_{\substack{S\in \Ind(M)\\ |S|=K}}\prod_{e\in S}x_e,
  \qquad x=(x_e)_{e\in E} .
\end{equation}
When $K\le \rank(M)$, this is the basis-generating polynomial of the rank-$K$ truncation of $M$.

\begin{theorem}[Matroid polynomial concavity]\label{thm:matroid-polynomial-concavity}
Let $M$ be a matroid and let $1\le K\le \rank(M)$. Then the function
\begin{equation}
\label{eq:hmk-definition}
  h_{M,K}(x)\coloneq f_{M,K}(x)^{1/K}
\end{equation}
is concave on $\R_{\ge 0}^{E}$.
Consequently, $f_{M,K}$ is quasi-concave on $\R_{\ge 0}^{E}$, i.e., for every $a\ge 0$, the superlevel set
\begin{equation}\label{eq:superlevel-set}
  \left\{x\in \R_{\ge 0}^{E}\colon f_{M,K}(x)\ge a\right\}
\end{equation}
is convex.
\end{theorem}

\begin{proof}
The rank-\(K\) truncation \(T_K(M)\) is a matroid, and \(f_{M,K}\) is its
basis-generating polynomial. By the complete log-concavity theorem for
matroid basis-generating polynomials due to Anari, Oveis Gharan, and
Vinzant~\cite{AnariOveisGharanVinzant2021}, \(f_{M,K}\) is completely log-concave. For a nonzero homogeneous polynomial $f$ of degree $d\ge2$ with nonnegative coefficients, log-concavity on the positive orthant is equivalent to concavity of $f^{1/d}$ on the same domain~\cite[Proposition~2.33]{BrandenHuh2020}. Applying this with $d=K$ gives the concavity of $h_{M,K}$ on the positive orthant, and continuity extends it to the closed nonnegative orthant. The case $K=1$ is immediate because $f_{M,1}$ is linear.

Since $h_{M,K}$ is concave, its superlevel sets are convex. Since
 $ f_{M,K}(x)=h_{M,K}(x)^K $
and $s\mapsto s^K$ is increasing for $s\ge 0$, the superlevel sets of $f_{M,K}$ coincide with those of $h_{M,K}$ after reparametrizing the threshold.
\end{proof}

Combining \eqref{eq:independence-polynomial-relation-intro}, \eqref{eq:fk-definition} and \eqref{eq:hmk-definition}, we have
\begin{equation}\label{eq:F-K-factorial-f}
  F_{M,K}(p)=K!\,f_{M,K}(p)=K!\,h^K_{M,K}(p).
\end{equation}
Thus maximizing $F_{M,K}$ is equivalent to maximizing $f_{M,K}$ or $h_{M,K}$ on the simplex $\Delta(E)$.

\section{Transitive matroids}
\label{sec:transitive-matroid-theorem}

The following theorem states that the uniform distribution always maximizes the probability that $K$ i.i.d.\ samples from $E$ are distinct and form an independent set of $M$, provided $M$ is transitive.

\begin{theorem}[Uniform optimality for transitive matroids]
\label{thm:transitive-matroid}
Let $M$ be a matroid on a finite ground set $E$, and let $1\le K\le \rank(M)$. Suppose that $\Aut(M)$ acts transitively on $E$. Let $u=(u_e)_{e\in E}$ with
\begin{equation}
\label{eq:uniform-distribution}
  u_e\coloneq\frac{1}{|E|},
  \qquad e\in E.
\end{equation}
Then, for every $p\in \Delta(E)$,
\begin{equation}
\label{eq:transitive-optimality-F}
  F_{M,K}(p)\le F_{M,K}(u).
\end{equation}
\end{theorem}

\begin{proof}
Let $G=\Aut(M)$ act on \(E\). This induces an action on
\(\Delta(E)\) by
\begin{equation}
\label{eq:group-action-on-distributions}
  (gp)_e=p_{g^{-1}e},\qquad g\in G,\ e\in E.
\end{equation}
Since the action of \(G\) on \(E\) is transitive, averaging any
\(p\in\Delta(E)\) over \(G\) gives the uniform distribution,
\begin{equation}\label{eq:group-average-is-uniform}
  \frac{1}{|G|}\sum_{g\in G}g p=u.
\end{equation}
Further, since automorphisms preserve independent sets, we have
\begin{equation}
\label{eq:invariance-f}
  f_{M,K}(g p)=f_{M,K}(p),
  \qquad g\in G.
\end{equation}
By concavity of $h_{M,K}=f_{M,K}^{1/K}$,
\begin{subequations}
\begin{align}
  h_{M,K}(u)
  &=h_{M,K}\left(\frac{1}{|G|}\sum_{g\in G}g p\right) \label{eq:concavity-average-start}\\
  &\ge \frac{1}{|G|}\sum_{g\in G}h_{M,K}(g p) \label{eq:concavity-average}\\
  &=\frac{1}{|G|}\sum_{g\in G}h_{M,K}(p) \label{eq:invariance-average}\\
  &=h_{M,K}(p). \label{eq:average-finish}
\end{align}
\end{subequations}
Raising to the $K$-th power and using \eqref{eq:F-K-factorial-f} gives \eqref{eq:transitive-optimality-F}.
\end{proof}

\begin{remark}[Non-transitive matroids]\label{rem:nontransitive-orbits}
The same proof shows that, for an arbitrary subgroup $G\le \Aut(M)$, there is a maximizer of $F_{M,K}$ that is constant on the $G$-orbits of $E$. Replacing $p$ by its group average cannot decrease $h_{M,K}$ (or $F_{M,K}$).
\myqed
\end{remark}

Notice that we have not claimed uniqueness or stability in Theorem~\ref{thm:transitive-matroid}. Indeed, in general, a transitive matroid may have many maximizers.

\begin{example}[A transitive matroid with nonunique maximizers]\label{ex:parallel-classes-nonunique}
Let $m\ge2$ be an integer and let $E=A\sqcup B$, where
\begin{equation}\label{eq:parallel-class-sizes}
  A=\{a_1,\ldots,a_m\},
  \qquad
  B=\{b_1,\ldots,b_m\}.
\end{equation}
Consider the rank-$2$ matroid $M$ in which the elements of $A$ are mutually parallel, the elements of $B$ are mutually parallel, and an element of $A$ is independent from an element of $B$.  Equivalently, the independent sets of size $2$ are precisely
\begin{equation}\label{eq:parallel-example-independent-pairs}
  \big\{\{a_i,b_j\} \colon 1\le i,j\le m\big\}.
\end{equation}
The automorphism group contains arbitrary permutations inside $A$ and $B$ and also the involution interchanging $A$ and $B$. Therefore, it acts transitively on $E$.

For $K=2$, if
\begin{equation}\label{eq:parallel-example-masses}
  p(A)\coloneq\sum_{a\in A}p_a,
  \qquad
  p(B)\coloneq\sum_{b\in B}p_b,
\end{equation}
then
\begin{equation}\label{eq:parallel-example-F}
  F_{M,2}(p)=2p(A)p(B).
\end{equation}
Thus $F_{M,2}$ is maximized when
\begin{equation}\label{eq:parallel-example-maximizers}
  p(A)=p(B)=\frac{1}{2}.
\end{equation}
The distribution inside each parallel class is arbitrary.  Therefore the uniform distribution is a maximizer but is not unique.  Consequently, no inequality of the form
\begin{equation}\label{eq:parallel-example-no-stability}
  F_{M,2}(u)-F_{M,2}(p)\ge c\|p-u\|_2^2
\end{equation}
can hold for all $p\in\Delta(E)$ with any $c>0$.
\myqed
\end{example}

The following statement gives an elementary criterion by which a local quadratic estimate can imply uniqueness and global quadratic stability. We will apply it first to simple transitive matroids and then to the explicit projective-geometry calculation in the next section.

\begin{proposition}[Local quadratic stability implies global quadratic stability]
\label{prop:local-to-global-stability}
Let $M$ be a matroid on a finite ground set $E$, let $1\le K\le \rank(M)$, and suppose that $\Aut(M)$ acts transitively on $E$.  Let $u$ be the uniform distribution on $E$.  Assume that there exist $\varepsilon>0$ and $c_{\mathrm{loc}}>0$ such that
\begin{equation}\label{eq:general-local-quadratic-stability}
  F_{M,K}(u)-F_{M,K}(p)\ge c_{\mathrm{loc}}\|p-u\|_2^2
\end{equation}
whenever $p\in\Delta(E)$ and $\|p-u\|_2\le\varepsilon$.  Then $u$ is the unique maximizer of $F_{M,K}$ on $\Delta(E)$, and there exists $c>0$ such that
\begin{equation}\label{eq:general-global-quadratic-stability}
  F_{M,K}(u)-F_{M,K}(p)\ge c\|p-u\|_2^2,
  \qquad p\in\Delta(E).
\end{equation}
\end{proposition}

\begin{proof}
By Theorems \ref{thm:matroid-polynomial-concavity} and \ref{thm:transitive-matroid}, $h_{M,K}(p)$ is concave on $\Delta(E)$ with a global maximizer $u$.

To show uniqueness, suppose that $p$ is a global maximizer, i.e., that $p$ and $u$ both maximize $h_{M,K}$.  Then, for all $t\in[0,1]$,
\begin{align}
\label{eq:general-uniqueness-concavity}
  h_{M,K}((1-t)p+tu)
  \ge (1-t)h_{M,K}(p)+t h_{M,K}(u)
  =h_{M,K}(u).
\end{align}
Since $h_{M,K}(u)$ is the global maximum, equality in \eqref{eq:general-uniqueness-concavity} holds throughout the segment from $p$ to $u$.  If $p\ne u$, then points on this segment sufficiently close to, but distinct from, $u$ contradict the local estimate \eqref{eq:general-local-quadratic-stability}.  Hence $p=u$.

For global stability, define
\begin{equation}\label{eq:general-R-ratio}
  R(p)\coloneq\frac{F_{M,K}(u)-F_{M,K}(p)}{\|p-u\|_2^2},
  \qquad p\ne u.
\end{equation}
The local estimate gives $R(p)\ge c_{\mathrm{loc}}$ near $u$. Define the compact set
\begin{equation}\label{eq:general-away-set}
  A_\varepsilon\coloneq\{p\in\Delta(E) \colon \|p-u\|_2\ge\varepsilon\}.
\end{equation}
If $A_\varepsilon=\varnothing$, then \eqref{eq:general-local-quadratic-stability} holds throughout $\Delta(E)$, giving \eqref{eq:general-global-quadratic-stability} with $c=c_{\mathrm{loc}}$. Otherwise, the numerator of $R$ is strictly positive on $A_\varepsilon$ by uniqueness, so $R$ attains a positive minimum there. Taking $c$ to be the smaller of this minimum and $c_{\mathrm{loc}}$ gives \eqref{eq:general-global-quadratic-stability}.
\end{proof}

We now show that Example~\ref{ex:parallel-classes-nonunique} essentially describes the only way non\-uniqueness can occur when $2\le K\le\rank(M)$. Recall that a matroid is \emph{simple} if it has no loops and no parallel elements~\cite{Oxley2011}. Its \emph{simplification} is obtained by deleting the loops and retaining one element from each parallel class.

\begin{theorem}[Maximizers and stability via parallel classes]
\label{thm:parallel-class-maximizers}
Let $M$ be a transitive matroid on $E$, let $2\le K\le\rank(M)$, and let $u$ be the uniform distribution on $E$. Denote the parallel classes of $M$ by $C_1,\ldots,C_s$. For $p\in\Delta(E)$, put $a_i=\sum_{e\in C_i}p_e$.
The distribution $p$ maximizes $F_{M,K}$ if and only if
\begin{equation}\label{eq:parallel-class-equality}
  a_i=\frac1s,\qquad 1\le i\le s.
\end{equation}
Moreover, there is a constant $c_{M,K}>0$ such that
\begin{equation}\label{eq:parallel-class-stability}
  F_{M,K}(u)-F_{M,K}(p)
  \ge c_{M,K}\sum_{i=1}^{s}\left(a_i-\frac1s\right)^2
  \qquad\forall p\in\Delta(E).
\end{equation}
In particular, $u$ is the unique maximizer if and only if $M$ is simple; in that case, \eqref{eq:parallel-class-stability} is a global quadratic stability bound in $\|p-u\|_2$.
\end{theorem}

\begin{proof}
First suppose that $M$ is simple, and put $n=|E|$ and $f=f_{M,K}$. Its rank-$K$ truncation remains simple because $K\ge2$. Applying \cite[Theorem~1.1(i)]{MuraiNagaokaYazawa2021} to this truncation shows that the Hessian $H=\nabla^2 f(u)$ has one positive and $n-1$ negative eigenvalues. Transitivity implies that all coordinates of $\nabla f(u)$ are equal. Euler's identity and homogeneity give
\begin{align}
  \nabla f(u)=Kf(u)\mathbf1,
  \qquad
  Hu=(K-1)\nabla f(u)=nK(K-1)f(u)u,
\end{align}
where $\mathbf1$ is the all-ones vector. Since $f(u)>0$, the vector $u$ spans the positive eigenspace of $H$. Hence $H$ is negative definite on the tangent space $\{v\in\R^E\colon\sum_e v_e=0\}=u^\perp$. On this space, $\nabla f(u)\cdot v=0$, so Taylor's theorem, together with $F_{M,K}=K!f$, gives the local estimate \eqref{eq:general-local-quadratic-stability}. Proposition~\ref{prop:local-to-global-stability} then gives uniqueness and global quadratic stability.

For general transitive $M$ of positive rank, there are no loops, and all parallel classes have the same size. Its automorphisms induce a transitive action on the simplification $\overline M$, whose ground set we identify with $\{1,\ldots,s\}$. An independent set contains at most one element from each parallel class, so, with $a=(a_1,\ldots,a_s)$,
\begin{equation}\label{eq:simplification-probability}
  F_{M,K}(p)=F_{\overline M,K}(a).
\end{equation}
The uniform distribution $u$ induces the uniform vector $(1/s,\ldots,1/s)$ on the classes. Applying the above reasoning for simple matroids to $\overline M$ proves \eqref{eq:parallel-class-equality} and \eqref{eq:parallel-class-stability}. If a class has more than one element, then its fixed total mass can be redistributed, giving distinct maximizers. Thus elementwise uniqueness holds exactly when $M$ is simple.
\end{proof}

Note that the (trivial) case $K=1$ is not covered by the statement: for every transitive matroid of positive rank, $F_{M,1}(p)=1$ for all $p\in\Delta(E)$.

\section{Projective geometries and random matrices over finite fields}
\label{sec:finite-fields}

Throughout this section, let $q$ be a prime power and let $N\ge1$ be an integer. Put
\begin{equation}
\label{eq:E-PG-section}
  E=\PG(N-1,q),
  \qquad
  m=|E|=\frac{q^N-1}{q-1},
  \qquad
  u_e=\frac{1}{m}.
\end{equation}
Thus $E$ is the set of one-dimensional subspaces of $\Fq^N$. To obtain an i.i.d.\ row model from $p\in\Delta(E)$, one can fix a nonzero representative of each projective point and draw the projective classes independently from $p$. The full-row-rank probability does not depend on the choice of representatives.

\begin{corollary}[Full-rank matrices over finite fields]
\label{cor:finite-field-full-rank}
Let $1\le K\le N$, and let $p$ be any probability distribution on $\PG(N-1,q)$. Then the probability that $K$ i.i.d.\ projective row classes drawn from $p$ have full row rank is maximized by the uniform distribution $u$ on $\PG(N-1,q)$.
\end{corollary}

\begin{proof}
The projective geometry matroid $\PG(N-1,q)$ has rank $N$. Its independent sets are precisely the projective point sets represented by linearly independent vectors in $\Fq^N$. The group $\operatorname{PGL}(N,q)$ acts transitively on projective points. Thus Theorem~\ref{thm:transitive-matroid} applies.
\end{proof}

The optimal value in this case is explicit.
Under the uniform distribution, after $j$ independent projective points have been drawn, their span contains
\begin{equation}\label{eq:projective-span-size}
  [j]_q\coloneq\frac{q^j-1}{q-1}
\end{equation}
projective points. Hence
\begin{equation}\label{eq:finite-field-optimal-value-projective}
  F_{\PG(N-1,q),K}(u)=\prod_{j=0}^{K-1}\left(1-\frac{[j]_q}{m}\right),
\end{equation}
where $[0]_q=0$ and $m$ is given in \eqref{eq:E-PG-section}. Equivalently,
\begin{equation}\label{eq:finite-field-optimal-value-vector}
  F_{\PG(N-1,q),K}(u)=\prod_{j=0}^{K-1}\frac{q^N-q^j}{q^N-1}.
\end{equation}

\begin{remark}[Distributions on vectors]\label{rem:nonzero-vector-distributions}
Let $P$ be a probability distribution on $\Fq^N$ and let $a=P(0)$. When $a<1$, define its conditional projective distribution by
\begin{equation}\label{eq:projective-pushforward}
  p_L\coloneq\frac{1}{1-a}\sum_{v\in L\setminus\{0\}}P(v),
  \qquad L\in\PG(N-1,q).
\end{equation}
For $1\le K\le N$, a matrix with i.i.d.\ rows drawn from $P$ can have full row rank only if every row is nonzero. Conditioning on this event gives
\begin{equation}\label{eq:full-rank-zero-mass}
  \Prob_P(\text{full row rank})=(1-a)^K F_{\PG(N-1,q),K}(p).
\end{equation}
For $a=1$, the probability is zero. For $0<a<1$, Corollary~\ref{cor:finite-field-full-rank} bounds \eqref{eq:full-rank-zero-mass} by $(1-a)^K F_{\PG(N-1,q),K}(u)$, which is strictly below the positive optimum in \eqref{eq:finite-field-optimal-value-vector}. Hence every maximizing distribution has $P(0)=0$. Consequently, any distribution with no mass at zero and uniform projective pushforward is globally optimal among all distributions on $\Fq^N$. The uniform distribution on $\Fq^N\setminus\{0\}$ is one such distribution, but it is not the only one when $q>2$.
\myqed
\end{remark}

For $2\le K\le N$, qualitative uniqueness and global quadratic stability on projective space already follow from Theorem~\ref{thm:parallel-class-maximizers}. We nonetheless give a direct argument below because it makes the Hessian and the local quadratic coefficient explicit. These estimates quantify the loss in full-rank probability when the projective distribution departs from uniformity.

\begin{proposition}[Exact equality and stability for $K=2$]\label{prop:K2-stability}
For $E=\PG(N-1,q)$ with $N\ge2$ and $K=2$, one has
\begin{equation}\label{eq:K2-formula}
  F_{\PG(N-1,q),2}(p)=1-\sum_{e\in E}p_e^2.
\end{equation}
Consequently,
\begin{equation}\label{eq:K2-exact-stability}
  F_{\PG(N-1,q),2}(u)-F_{\PG(N-1,q),2}(p)=\sum_{e\in E}(p_e-u_e)^2.
\end{equation}
In particular, equality holds if and only if $p=u$.
\end{proposition}

\begin{proof}
Since the projective geometry matroid is simple, two sampled projective points are independent if and only if they are distinct. Therefore
\begin{equation}
\label{eq:K2-distinct-probability}
  F_{\PG(N-1,q),2}(p)=1-\Prob(X_1=X_2)=1-\sum_{e\in E}p_e^2.
\end{equation}
Since $\sum_e p_e=\sum_e u_e=1$,
\begin{equation}
\label{eq:norm-expansion}
  \sum_{e\in E}(p_e-u_e)^2=\sum_{e\in E}p_e^2-\frac{1}{m}.
\end{equation}
Also
\begin{equation}
\label{eq:K2-uniform-value}
  F_{\PG(N-1,q),2}(u)=1-\frac{1}{m}.
\end{equation}
Combining \eqref{eq:K2-distinct-probability}, \eqref{eq:norm-expansion}, and \eqref{eq:K2-uniform-value} gives \eqref{eq:K2-exact-stability}.
\end{proof}

The preceding exact stability identity has the following higher-rank
analogue, with a possibly non-explicit stability constant.

\begin{theorem}[Uniqueness and global quadratic stability]\label{thm:global-stability-PG}
Let $E=\PG(N-1,q)$ and let $2\le K\le N$. There exists a constant
 $ c_{N,q,K}>0 $
such that, for every $p\in\Delta(E)$,
\begin{equation}\label{eq:global-stability-inequality}
  F_{\PG(N-1,q),K}(u)-F_{\PG(N-1,q),K}(p)
  \ge c_{N,q,K}\|p-u\|_2^2.
\end{equation}
Consequently, $u$ is the unique maximizer of $F_{\PG(N-1,q),K}$ on $\Delta(E)$.
\end{theorem}

\begin{proof}
We demonstrate local quadratic stability at $u$. Uniqueness and global quadratic stability will then follow from Proposition~\ref{prop:local-to-global-stability}.
Denote
\begin{equation}\label{eq:stability-f-definition}
  f=f_{\PG(N-1,q),K},
  \qquad
  F=F_{\PG(N-1,q),K}=K!f.
\end{equation}
Write
\begin{equation}\label{eq:v-definition}
  v\coloneq p-u,
  \qquad
  \sum_{e\in E}v_e=0,
\end{equation}
and let
\begin{equation}
\label{eq:tangent-space}
  T\coloneq\left\{v\in \R^E \colon \sum_{e\in E}v_e=0\right\}.
\end{equation}
The point $u$ is a critical point of $F$ restricted to the affine hyperplane $\sum_e p_e=1$, because it is fixed by the transitive action of $\operatorname{PGL}(N,q)$ and is a global maximizer by Corollary~\ref{cor:finite-field-full-rank}.

Since $f$ is multi-affine,
\begin{equation}
\label{eq:diagonal-second-derivatives}
  \frac{\partial^2 f}{\partial x_e^2}(u)=0.
\end{equation}
For distinct projective points $e\ne e'$, every pair $\{e,e'\}$ is independent, and the projective linear group is transitive on ordered pairs of distinct projective points. Hence the mixed second derivative is independent of the pair. Let
\begin{equation}
\label{eq:B2-definition}
  B_2\coloneq\#\left\{S\subseteq E\colon\ |S|=K,\ S\text{ independent},\ e,e'\in S\right\},
\end{equation}
where $e\ne e'$ is fixed. Then
\begin{equation}\label{eq:mixed-second-derivative}
  \frac{\partial^2 f}{\partial x_e\partial x_{e'}}(u)=B_2 m^{-(K-2)}
  \qquad (e\ne e').
\end{equation}
Moreover $B_2>0$ because $K\le N$. Thus, for $v\in T$,
\begin{subequations}
\begin{align}
  v^T \nabla^2 f(u)v
  &=B_2m^{-(K-2)}\sum_{e,e'\in E,\, e\ne e'}v_ev_{e'} \label{eq:hessian-quadratic-start}\\
  &=B_2m^{-(K-2)}\left(\left(\sum_{e\in E}v_e\right)^2-\sum_{e\in E}v_e^2\right) \label{eq:hessian-quadratic-middle}\\
  &=-B_2m^{-(K-2)}\|v\|_2^2. \label{eq:hessian-quadratic-end}
\end{align}
\end{subequations}
Consequently,
\begin{equation}\label{eq:F-hessian-negative}
  v^T\nabla^2 F(u)v=-K!B_2m^{-(K-2)}\|v\|_2^2.
\end{equation}
Taylor's theorem on the affine hyperplane $\sum_e p_e=1$ gives
\begin{equation}\label{eq:taylor-local-stability}
  F(u+v)=F(u)-\frac{K!B_2m^{-(K-2)}}{2}\|v\|_2^2+O(\|v\|_2^3).
\end{equation}
Therefore there exist $\varepsilon>0$ and $c_{\mathrm{loc}}>0$ such that
\begin{equation}\label{eq:local-stability-inequality}
  F(u)-F(p)\ge c_{\mathrm{loc}}\|p-u\|_2^2
\end{equation}
whenever $\|p-u\|_2\le \varepsilon$.
\end{proof}

The number $B_2$ in \eqref{eq:B2-definition} can be written explicitly. With $[j]_q$ as in \eqref{eq:projective-span-size},
\begin{equation}\label{eq:B2-explicit}
  B_2=\frac{1}{(K-2)!}\prod_{j=2}^{K-1}\left(m-[j]_q\right),
\end{equation}
with the empty product interpreted as $1$ when $K=2$.
(Starting with the fixed pair $e,e'$, there are $m-[j]_q$ choices for the next point after $j$ independent points have been selected, and each resulting $K$-set is counted $(K-2)!$ times.)

\begin{remark}[Vector-level equality]\label{rem:vector-level-equality}
Let $2\le K\le N$. By Theorem~\ref{thm:global-stability-PG} and \eqref{eq:full-rank-zero-mass}, a distribution $P$ on $\Fq^N$ maximizes the full-row-rank probability if and only if $P(0)=0$ and its conditional projective distribution \eqref{eq:projective-pushforward} is uniform. Equivalently, one must have
\begin{equation}\label{eq:vector-level-equality-condition}
  \sum_{v\in L\setminus\{0\}}P(v)=\frac{1}{|\PG(N-1,q)|}=\frac{q-1}{q^N-1}
\end{equation}
for every projective point $L\in\PG(N-1,q)$. For $q>2$, the maximizer is nonunique because mass can be redistributed within each class of vectors which are nonzero scalar multiples of each other. For $q=2$, each class contains exactly one nonzero vector, so the uniform distribution on $\Fq^N\setminus\{0\}$ is the unique maximizer. When $K=1$, every distribution with $P(0)=0$ is optimal.
\myqed
\end{remark}

\section{Sampled-rank distributions}
\label{sec:further-questions}

The uniform-optimality principle extends from independence to every upper tail of the sampled-rank distribution. Denote by $\rank_M(X_1,\ldots,X_L)$ the rank of the set $\{X_1,\ldots,X_L\}$ (so repetitions do not increase rank).

\begin{theorem}[Uniform optimality of rank tails]\label{thm:rank-tail-optimality}
Let $M$ be a finite matroid on $E$, and let $L\ge1$ and $r\ge1$ be integers such that $r\le\min\{L,\rank(M)\}$. For i.i.d.\ samples from $p\in\Delta(E)$, define
\begin{equation}
\label{eq:rank-threshold-question}
  Q_{M,L,r}(p)\coloneq\Prob_p\bigl(\rank_M(X_1,\ldots,X_L)\ge r\bigr).
\end{equation}
The homogeneous degree-$L$ polynomial representing $Q_{M,L,r}$ has a concave $L$-th root on $\R_{\ge0}^{E}$. If $M$ is transitive and $u$ is uniform on $E$, then
\begin{equation}\label{eq:rank-tail-uniform-bound}
  Q_{M,L,r}(p)\le Q_{M,L,r}(u)
  \qquad\text{for every }p\in\Delta(E).
\end{equation}
\end{theorem}

\begin{proof}
For a vector $\alpha\in\mathbb Z_{\ge0}^{E}$, write $|\alpha|=\sum_e\alpha_e$, $\alpha!=\prod_e\alpha_e!$, $x^\alpha=\prod_e x_e^{\alpha_e}$, and $\operatorname{supp}\alpha=\{e\colon\alpha_e>0\}$. Let
\begin{equation}
  \mathcal J=\bigl\{\alpha\in\mathbb Z_{\ge0}^{E}\colon |\alpha|=L,
  \ \rank_M(\operatorname{supp}\alpha)\ge r\bigr\}.
\end{equation}
The sample multiplicities have a multinomial distribution, so the required homogeneous polynomial is
\begin{equation}\label{eq:rank-tail-polynomial}
  Q_{M,L,r}(x)=L!\sum_{\alpha\in\mathcal J}\frac{x^\alpha}{\alpha!}.
\end{equation}
To prove its root-concavity, consider
\begin{equation}
  g(x)=f_{M,r}(x)\left(\sum_{e\in E}x_e\right)^{L-r}.
\end{equation}
The polynomial $f_{M,r}$ is the basis-generating polynomial of a truncation of $M$, and hence is Lorentzian by \cite[Theorem~3.10]{BrandenHuh2020}. The linear polynomial $\sum_e x_e$ is also Lorentzian. Repeated application of the product property~\cite[Corollary~2.32]{BrandenHuh2020} shows that $g$ is Lorentzian, with the second factor interpreted as $1$ when $L=r$. Its monomial support is exactly $\mathcal J$: a monomial $x^\alpha$ occurs precisely when $\operatorname{supp}\alpha$ contains an independent $r$-set. By \cite[Theorem~3.10]{BrandenHuh2020}, the normalized support polynomial $\sum_{\alpha\in\mathcal J}x^\alpha/\alpha!$ is therefore Lorentzian. Thus $Q_{M,L,r}$ is log-concave by \cite[Theorem~2.30]{BrandenHuh2020}, and its $L$-th root is concave by \cite[Proposition~2.33]{BrandenHuh2020}, as in Theorem~\ref{thm:matroid-polynomial-concavity}. The case $L=1$ is linear, and continuity again includes the boundary.

Automorphisms preserve rank, so $Q_{M,L,r}(gp)=Q_{M,L,r}(p)$. When $M$ is transitive, \eqref{eq:group-average-is-uniform} and concavity give
\begin{equation}
  Q_{M,L,r}(u)^{1/L}
  \ge\frac1{|\Aut(M)|}\sum_{g\in\Aut(M)}Q_{M,L,r}(gp)^{1/L}
  =Q_{M,L,r}(p)^{1/L},
\end{equation}
which proves \eqref{eq:rank-tail-uniform-bound}.
\end{proof}

The other thresholds are trivial: the tail probability is $1$ for $r\le0$ and $0$ for $r>\min\{L,\rank(M)\}$. Hence, for a transitive matroid, sampled rank under uniform sampling first-order stochastically dominates sampled rank under every other i.i.d.\ sampling distribution. In particular,
\begin{equation}\label{eq:expected-rank-question}
  \E_p\,\rank_M(X_1,\ldots,X_L)
  =\sum_{r=1}^{\min\{L,\rank(M)\}}Q_{M,L,r}(p)
  \le\E_u\,\rank_M(X_1,\ldots,X_L).
\end{equation}
For $L=r=K$, Theorem~\ref{thm:rank-tail-optimality} recovers Theorem~\ref{thm:transitive-matroid}; the earlier direct proof avoids the additional support argument.

\section*{Funding}
This work was supported by the Ministry of Science, Technological Development and Innovation of the Republic of Serbia (contract no. 451-03-34/2026-03/200156) and by the Faculty of Technical Sciences, University of Novi Sad, Serbia (project no. 01-3609/1).

\section*{Declaration of competing interest}
The author declares that he has no known competing financial interests or personal relationships that could have appeared to influence the work reported in this paper.

\section*{Data availability}
No datasets were generated or analysed during the current study.

\end{document}